\documentclass[11pt]{amsart}
\usepackage{amssymb,amsmath,amsfonts,amsthm}
\usepackage{graphicx}

\newtheorem{theorem}{Theorem}[section]
\newtheorem{lemma}[theorem]{Lemma}

\theoremstyle{definition}
\newtheorem{definition}[theorem]{Definition}

\newtheorem{remark}[theorem]{Remark}

\newtheorem{ltheorem}{Theorem} 

\def\real{\mathbb{R}}

\def\integer{\mathbb{Z}}

\def\supp{\operatorname{supp}}
\def\d{\operatorname{dist}}
\def\dim{\operatorname{dim}}
\def\id{\operatorname{id}}
\def\cF{\mathcal{F}}

\def\rels{\sim^s}
\def\relu{\sim^u}

\def\cH{\mathcal{H}}

\def\cP{\mathcal{P}}

\def\cM{\mathcal{M}}
\def\cN{\mathcal{N}}

\def\quand{\quad\text{and}\quad}

\def\hA{\hat{A}}

\def\proj{\mathbb{R}P^{2d-1}}
\def\SL{SL(2,\real)}
\def\Sp{Sp(2d,\real)}
\def\GL{GL(2d,\real)}

\newcommand{\norm}[1]{{\left\lVert  #1  \right\rVert}}

\title[Positive Lyapunov exponents]
{Positive Lyapunov exponents for cocycles over time one anosov maps}
\author{Mauricio Poletti}
\address{LAGA -- Universit\'e Paris 13, 99 Av. Jean-Baptiste Cl\'ement, 93430 Villetaneus, France.}
\email{mpoletti@impa.br}

\begin{document}

\begin{abstract}
We prove that in an open and dense set, Symplectic linear cocycles over time one maps of Anosov flows, have positive Lyapunov exponents for SRB measures. 
\end{abstract}

\maketitle

\section{introduction}
Positiveness of Lyapunov exponents was widely studied in the past years.
Avila \cite{Av11} proved for several topologies, including $C^r$ topology for $r\geq 0$, 
that there exists a dense, but not open, set of cocycles taking values in $\SL$ with non zero Lyapunov exponents. This was generalized to $\Sp$ cocycles by Xu \cite{Xu15}.

A result stated by Ma\~n\'e and proved by Bochi \cite{Boc02} expose that the generic behaviour in the $C^0$ topology is not positiveness of Lyapunov exponents:
$C^0$ generically $\SL$ cocycles are uniformly hyperbolic or have zero Lyapunov exponents.

In the case of more regular cocycles with some hyperbolic behaviour on the base map and bunching condition on the fibers, the generic behavior changes radically. Viana~\cite{Almost} proved that in  this class the Lyapunov exponents are generically
non-zero: When the base maps are non-uniformly hyperbolic and the cocycle take values in the group $SL(d,\real)$. This was generalized by Bessa-Bochi-Cambrainha-Matheus-Varandas-Xu~\cite{BBCMVX} for any non-compact semi-simple Lie group (for example $\Sp$). 

These results were extended in the partially hyperbolic setting in two cases:
when the center leaves are compact, the map has zero exponent on the center
direction and the cocycle takes 
values in $\Sp$, by the author~\cite{Pol16}, 
and when the map is volume preserving dynamically coherent, accessible and the cocycle takes values in $SL(d,\real)$ by Avila-Viana-Santamaria~\cite{ASV13}.

Here we deal with an especial case of partially hyperbolic maps, the time one map of a Anosov flow that preserves an SRB measure. As we do not need to assume 
that our map is accessible, we can deal with cases that can not be covered by \cite{ASV13}, for example Anosov flows given by suspensions of hyperbolic diffeomorphism.

For an volume preserving Anosov flow, the volume is an SRB invariant measure, but we can also deal with dissipative flows (this means that do not preserve volume). For example any non-trivial attractor of an Anosov flows admits some SRB measure and our result cover this cases.

\medskip
\textbf{Acknowledgements.} Thanks to C. Matheus and M. Shannon for the discussions and useful ideas, and F. Lenarduzzi for the commentaries.

\section{Definitions and results}

A diffeomorphism $f:M \to M$ of a compact manifold $M$ is \emph{partially hyperbolic} if there
exists a non-trivial splitting of the tangent bundle
\begin{equation*}
TM=E^s\oplus E^c\oplus E^u
\end{equation*}
invariant under the derivative $Df$, a Riemannian metric $\|\cdot\|$ on $M$, and positive
constants $\lambda$, $\gamma$ with $\lambda<1$,
and $\lambda<\gamma$ such that, for any unit vector $v\in T_pM$,
\begin{alignat}{2}
& \|Df(p)v \| < \lambda & \quad & \text{if } v\in E^s(p),\label{eq.PH1}\\
\gamma < & \|Df(p)v \|  < \gamma^{-1} & & \text{if } v\in E^c(p), \label{eq.PH2}\\
{\lambda}^{-1} < & \|Df(p)v\| & &  \text{if } v\in E^u(p).\label{eq.PH3}
\end{alignat}
All three sub bundles $E^s$, $E^c$, $E^u$ are assumed to have positive dimension.

The stable and unstable bundles $E^s$ and $E^u$
are uniquely integrable and their integral manifolds form two transverse continuous foliations
$W^s$ and $W^u$, whose leaves are immersed sub manifolds of the same class of
differentiability as $f$. These foliations are referred to as the \emph{strong-stable} and
\emph{strong-unstable} foliations. They are invariant under $f$, in the sense that
$$
f(W^s(x))= W^s(f(x)) \quand f(W^u(x))= W^u(f(x)),
$$
where $W^s(x)$ and $W^s(x)$ denote the leaves of $W^s$ and $W^u$, respectively,
passing through any $x\in M$.

We take $M$ to be endowed with the distance $\d$ associated to such a Riemannian structure.
The \emph{Lebesgue class} is the measure class of the volume induced by this (or any other)
Riemannian metric on $M$. These notions extend to any sub manifold of $M$, just considering
the restriction of the Riemannian metric to the sub manifold. We say that $f$ is \emph{volume preserving}
if it preserves some probability measure in the Lebesgue class of $M$.

A special class of partially hyperbolic maps are the \emph{time one maps of Anosov flows}. We say that a flow $\phi_t:M\to M$ is an Anosov flow if there exists a splitting 
$$
TM=E^s\oplus X\oplus E^u,
$$
where $X(p)=\frac{d }{dt}\phi_t(p)\mid_{t=0}$, invariant under the derivative $D\phi_t$, where $E^s$ is contracted and $E^u$ is expanded by the derivative of $\phi_t$.

Let $f:M\to M$ be the time one map of a $C^2$ Anosov flow (i.e: $f=\phi_1$). This map is partially hyperbolic with center bundle $E^c=X$.

In this particular case the center bundle $E^c=X$ is integrable and $C^1$ (observe that $W^c(p)=\phi_\real(p)$).  
The bundles $E^c+E^u$ and $E^s+E^c$ are integrable and absolutely continuous (See \cite{3flows}). We call them center stable $W^{cs}$ and center unstable $W^{cu}$ foliations.

We say that a measure $\mu$ is an \emph{SRB measure}
for the flow if the disintegration of the measure along the center unstable leaves is in the Lebesgue class along this sub-manifolds (for a more detaills see \cite{3flows}). Observe that we are asking for the measure to be invariant for the flow not just for the time one map.

Recall that by the spectral decomposition theorem the non-wandering set of the flow can be decomposed into basic sets $\Omega(\phi_t)=\Delta_1\cup\cdots \cup \Delta_n$, where the $\Delta_i$ are invariant and the restriction $\phi_t$ is transitive in this sets, in particular the support of every ergodic measure is supported in one of this sets.

Also this basic sets coincide with the closure of the homoclinic class of its periodic points, as the support of an SRB measures is saturated by the center unstable foliation this implies that $\supp(\mu)=\Delta_i$ for some $i\in \{1,\dots,n\}$, moreover this set is an attractor.
Conversely, every non trivial attractor of an hyperbolic flow admits some  SRB measure.

The symplectic group $\Sp$ is the subgroup of matrices $B\in\GL$ such that 
\begin{equation*}
B^T J B=J,\quand J=\left(\begin{array}{cc}
0& \id_d\\ -\id_d&0
\end{array}\right).
\end{equation*}

Let $H^\alpha(M)$ be the Banach space of $\alpha$-H\"older maps $A:M\to \Sp$ with norm 
$$
\norm{A}_\alpha=\sup_{x\in M}\norm{A(x)}+\sup_{x\neq y}\frac{\norm{A(x)-A(y)}}{\norm{x-y}^\alpha}.
$$

The \emph{linear cocycle} defined by $A:M\rightarrow \Sp$ over 
$f:M\rightarrow M$ is the (invertible) map $F_A:M\times \real^{2d}\rightarrow M\times \real^{2d}$ 
given by
\begin{equation}
\nonumber F_A\left(x,v\right)=\left(f(x),A(x)v\right).
\end{equation}
Its iterates are given by $F^n_A\left(x,v\right)=\left(f^n(x),A^n(x)v\right)$ where
\begin{equation*}\label{def:cocycles}
A^n(x)=
\left\{
	\begin{array}{ll}
		A(f^{n-1}(x))\ldots A(f(x))A(x)  & \mbox{if } n>0 \\
		Id & \mbox{if } n=0 \\
		A(f^{n}(x))^{-1}\ldots A(f^{-1}(x))^{-1}& \mbox{if } n<0 \\
	\end{array}
\right.
\end{equation*}
Sometimes we denote this cocycle by $(f,A)$.

Let $\mu$ be an $f$-invariant probability measure on $M$ and suppose that $\log\norm{A}$ and $\log\norm{ A^{-1}}$
are integrable. By Kingman's sub-additive ergodic theorem, see \cite{FET}, 
the limits 
$$\begin{aligned}
   \lambda^+(A,\mu,x)&=\lim_{n\to \infty}\frac{1}{n}\log \norm{A^n(x)}\quand \\ 
   \lambda^-(A,\mu,x)&=\lim_{n\to \infty}-\frac{1}{n}\log \norm{(A^n(x))^{-1}},
  \end{aligned}
$$
exist for $\mu$-almost every $x\in M$. When there is no risk of ambiguity we write just 
$\lambda^+(x)=\lambda^+(A,\mu,x)$.

By Oseledets~\cite{Ose68}, at $\mu$-almost every point
$x\in M$ there exist real numbers $\lambda^+(x)=\lambda_1 (x)>\cdots >\lambda_k(x)=\lambda^-(x)$
and a decomposition $\real^d=E^1_{x} \oplus \cdots \oplus E^k_{x}$ into vector subspaces such that
\begin{equation*}
A(x)E^i_{x}=E^i_{f(x)}
\text{ and } \lambda_i(x)=\lim_{\mid n\mid \to \infty}\frac{1}{n}\log\|A^n(x)v\|
\end{equation*}
for every non-zero $v\in E^i_{x}$ and $1\leq i \leq k$.

Define $L(A,\mu)=\int \lambda^+ \,d\mu$, we say that $A$ \emph{has positive exponent} if
 $L(A,\mu)>0$. When $\mu$ is ergodic, as we are going to assume later, we have 
$L(A,\mu)=\lambda^+(x)$ for $\mu$-almost every $x\in M$.

We say that the cocycle is \emph{fiber bunched} if there exist $C>0$
and $\theta <1$ such that
\begin{equation}\label{eq.FB}
\norm{A^n(p)} \norm{A^n(p)^{-1}} \lambda ^{n \alpha} \leq C \theta ^n
\quad\text{for every $p\in M$ and $n\geq 0$,}
\end{equation}
where $\lambda$ is a hyperbolicity constant for $f$ as in conditions \eqref{eq.PH1}, \eqref{eq.PH2}, \eqref{eq.PH3}.

So now we can state the main theorem of this work.
\begin{ltheorem}\label{teo1}
Let $f:M\to M$ be a $C^2$ time one map of an Anosov flow and $\mu$ an SRB invariant measure for the flow, then $L(A,\mu)>0$ in an open and dense subset of the fiber bunched cocycles of $H^\alpha(M)$.
\end{ltheorem}

\section{Holonomies}
Suppose that $f:M\to M$ is partially hyperbolic. 
 The stable and unstable foliations are, usually, \emph{not} transversely smooth:
the holonomy maps between any pair of cross-sections are not even Lipschitz continuous,
in general, although they are always $\gamma$-H\"older continuous for some $\gamma>0$.
Moreover, if $f$ is $C^2$ then these foliations are absolutely continuous, in the following
sense.

Let $d=\dim M$ and $\cF$ be a continuous foliation of $M$ with $k$-dimensional smooth
leaves, $0<k<d$. Let $\cF(p)$ be the leaf through a point $p\in M$ and $\cF(p,R)\subset\cF(p)$
be the neighbourhood of radius $R>0$ around $p$, relative to the distance defined by the
Riemannian metric restricted to $\cF(p)$. A \emph{foliation box} for $\cF$ at $p$ is the
image of an embedding
$$
\Phi:\cF(p,R) \times \real^{d-k} \to M
$$
such that $\Phi(\cdot,0)=\id$, every $\Phi(\cdot, y)$ is a diffeomorphism from $\cF(p,R)$
to some subset of a leaf of $\cF$ (we call the image a \emph{horizontal slice}),
and these diffeomorphisms vary continuously with $y \in \real^{d-k}$.
Foliation boxes exist at every $p\in M$, by definition of continuous foliation with smooth leaves.
A \emph{cross-section} to $\cF$ is a smooth co-dimension-$k$ disk inside a foliation box that
intersects each horizontal slice exactly once, transversely and with angle uniformly bounded from zero.

Then, for any pair of cross-sections $\Sigma$ and $\Sigma'$, there is a well defined \emph{holonomy map}
$\Sigma\to\Sigma'$, assigning to each $x\in\Sigma$ the unique point of intersection of $\Sigma'$
with the horizontal slice through $x$. The foliation is \emph{absolutely continuous} if all these
homeomorphisms map zero Lebesgue measure sets to zero Lebesgue measure sets.
That holds, in particular, for  the strong-stable and strong-unstable foliations of partially hyperbolic
$C^2$ diffeomorphisms and, in fact, the Jacobians of all holonomy maps are bounded by a uniform constant.

Let $f:M\to M$ be the time one map of an $C^2$ Anosov flow, his center stable and center unstable foliations are also absolutely continuous.
If we take two center manifolds $W^c(p)$ and $W^c(z)$ in the same strong stable manifold, from every point $t\in W^c(p)$ we can define the stable holonomy locally $h^s_{p,z}:I\subset W^c(p)\to W^c(z)$ and the Jacobian of $h^s_{p,z}$ vary continuously with the points $p$ and $z$. Actually the Jacobian is given by 
$$
\lim_{n\to\infty}\frac{Jf^n_c(t)}{Jf^n_c(h^s_{p,z}(t))},
$$
where $Jf^n_c(t)=\det Df^n\mid_{E^c_t}(t)$.
Analogously for the unstable holonomies.
\subsection{Measure disintegration}

Given a measurable partition $\cP$ of $M$, by Rokhlin disintegration theorem (see \cite{FET}) there exists a family of measures $\{\mu_P\}_{P\in \cP}$ such that for every measurable set $B\in M$ 
\begin{itemize}
\item $P\to \mu_P(B)$ is measurable,
\item $\mu_P(P)=1$ and 
\item $\mu(B)=\int_M \mu_P (B) d\tilde{\mu}(P)$.
\end{itemize}
Moreover, such disintegrantion is essentially unique \cite{Rok62}.

In general the partition by the invariant foliations is not measurable, to overcome this problem we disintegrate our measure locally as we explain next.

We call $C_x\in M$ a foliated box centred in $x\in M$ if there exists a continuous function 
$\Phi:[-1,1]^d\to C$ such that 
\begin{itemize}
\item $\Phi(0)=x$,
\item for every $y\in [-1,1]^{d-1}$, 
$$
W^c_{C_x}(\Phi(y,0)):=\Phi(y,[-1,1])\subset W^c(\Phi(y,0)),
$$
\item for every $y\in [-1,1]^u$, 
$$W^{cs}_{C_x}(\Phi(0_s,y,0)):=\Phi([-1,1]^s\times \{y\} \times [-1,1])\subset W^{cs}(\Phi(0_s,y,0)),$$
\item  for every $x\in [-1,1]^s$, 
$$
W^{cu}_{C_x}(\Phi(x,0_u,0)):=\Phi(\{x\}\times [-1,1]^u \times [-1,1])\subset W^{cu}(\Phi(x,0_u,0)),
$$
\end{itemize}
where $s$ is the dimension of $E^s$ and $u$ the dimension of $E^u$. 

So $\mu$ is an SRB measure if and only if for every $x$ the disintegration of $\mu\mid C_x$ given by the partition $\{W^{cu}_{C_x}(y)\}_{y\in C_x}$ gives measures absolutely continuous with respect to the Lebesgue measure on $W^{cu}(y)$. 

Given a foliated box $C_x$ lets call $\Sigma_x=\Phi([-1,1]^{d-1}\times\{0\})$ and $\pi:C_x\to \Sigma_x$ the natural projection given by $\pi(\Phi(x_s,y_u,t))=\Phi(x_s,y_u,0)$, observe that $\Sigma_x$ has a product structure $\Sigma^s\times \Sigma^u$ given by $\Phi$.
We say that $\mu$ has \emph{projective product structure} if $\pi_*\mu\mid C_x$ is absolutely continuous with a product measure $\mu^s\times \mu^u$, where $\mu^*$ is a measure on $\Sigma^*$ for $*=s,u$.
 
We say that a measure $\mu$ has a \emph{good product structure} if it has projective product structure and also for every foliated box the disintegration in center manifolds $\{W^c_{C_x}(y)\}_{y\in C_x}$ is absolutely continuous with respect to the Lebesgue measure.

Every SRB measure for the flow has a good projective product structure, the absolute continuity in the center manifold is just because the measure is invariant by the flow and the projective product structure is because of the absolute continuity of the center stable foliation (see \cite[Proposition~3.4]{ViY13} for a proof  of this fact).

The good product structure of the SRB measure will be the key property to prove theorem~\ref{t.sucinv} in section~\ref{s.proof}, that is the main tool for proving our result. 
\subsection{Linear holonomies}
\begin{definition}
We say that $A$ admits \textit{strong stable holonomies} if there exist, for every $p\textrm{ and }q \in M$ with $p\in W^{s}_{loc}(q)$,
 linear transformations $H^{s,A}_{p,q}:\real^{2d}\to \real^{2d}$ 
with the properties:
\begin{enumerate}
\item $H^{s,A}_{f^j(p),f^j(q)}=A^j(q)\circ H^{s,A}_{p,q}\circ A^j(p)^{-1}$ for every $j\geq 1$,\label{prop.holon}
\item $H^{s,A}_{p,p}=id$ and, for $w\in W^s(p)$, $H^{s,A}_{p,q}=H^{s,A}_{w,q}\circ H^{s,A}_{p,w}$,
\item there exists some $L>0$ (that does not depend on $p,q$) such that $\lVert H^{s,A}_{p,q}-id \rVert \leq L\d(p,q)^{\alpha}$.
\end{enumerate}
Moreover, these holonomies vary continuously with the base points $p$ and $q$.

These linear transformations are called \textit{strong stable holonomies}.
\end{definition}

Observe that if $p$ and $q$ are in the same strong stable manifold (not necessarily local) there exists some $n\geq 0$ such that $f^n(p)$ and $f^n(q)$ are in the same local strong stable manifold. So we can extend the holonomies to these points defining $H^{s,A}_{p,q}=A^n(q)^{-1}\circ H^{s,A}_{f^n(p),f^n(q)}\circ A^n(p)$, by property \eqref{prop.holon} this does not depend on the choice of $n$.

Analogously if $p\in W^{u}_{loc}(q)$ we have the \textit{strong unstable holonomies}. When there is no ambiguity we write $H^{s}_{p,q}$ instead of $H^{s,A}_{p,q}$.

\begin{remark}
The fiber bunched condition \eqref{eq.FB} implies the existence of the strong stable and strong unstable holonomies (see \cite{ASV13}). 
\end{remark}

Let $\proj$ be the real projective space, i.e: the quotient space of $\real^{2d}\setminus\{0\}$ by the equivalence 
relation $v\sim u$ if there exist $c\in\real\setminus\{0\}$, such that, $v=c u$. Every invertible linear transformation $B$ induces a
projective transformation, that we also denote by $B$, 
$$B:\proj\to \proj \quand B\left[v\right]=\left[B(v)\right].$$ 
We also denote by $F:M\times \proj \to M\times \proj$ the induced projective cocycle.
\subsection{Invariance principle}
Let $P:M\times \proj\to M$ be the projection to the first coordinate, $P(x,v)=x$, and let $m$ be a $F$-invariant measure on 
$M\times \proj$ such that $P_{\ast}m=\mu$.
By Rokhlin \cite{Rok62} we can disintegrate the measure $m$ into $\lbrace m_x\mbox{, }x \in M\rbrace$ with
 respect to the partitions 
 $$
 \mathcal{P}=\lbrace \{p\}\times \proj\mbox{, }p\in M\rbrace,
 $$

A measure $m$ that projects on $\mu$ is called:
\begin{itemize}
 \item[($a$)] $u$-\emph{invariant} if there exists a full $\mu$-measure subset $M^s\subset M$ such that $m_z=(H^s_{y,z})_{\ast}m_y$ for every $y,z\in M^s$ in the same strong-stable leaf;
 \item[($b$)] $s$-\emph{invariant} if there exists a full $\mu$-measure subset $M^u\subset M$ such that $m_z=(H^u_{y,z})_{\ast}m_y$ for every $y,z\in M^u$ in the same strong-unstable leaf.
\end{itemize}
If both $a$ and $b$ are satisfied we say that $m$ is $su$-\emph{invariant}.

We recall the following result whose proof can be found in \cite{ASV13}.
\begin{theorem}\label{t.IP}
Let $f:M\to M$ be a $C^1$ partially hyperbolic diffeomorphism and $A:M\mapsto \GL$ a linear cocycle defined over $f$. If $\mu$ is an $f$-invariant measure, such that $\lambda^+(A,\mu,x)=\lambda^-(A,\mu,x)=0$ almost everywhere, then every $F_A$-invariant measure $m$ such that $P_* m=\mu$ is $su$-invariant.
\end{theorem}

\section{Proof of the main result}\label{s.proof}
We say that a disintegration of $m$ is \emph{$u/c$-invariant} if, for every $x$ and $y$ in the support of $\mu$ and in the same unstable manifold, for Lebesgue almost every $t\in W^c(x)$, ${H^u_{t,h(t)}}_* m_t=m_{h(t)}$ where $h(t)=W^c(y)\cap W^u(t)$.  
Analogously, we call a disintegration \emph{$s/c$-invariant} changing unstable by stable. We say that the disintegration is \emph{$su/c$-invariant} if the disintegration is both $u/c$ and $s/c$-invariant.

The primary tool to prove our main theorem is given by the next theorem:
\begin{theorem}\label{t.sucinv}
Let $f:M\to M$ be a $C^2$ partially hyperbolic, dynamically coherent diffeomorphism and $\mu$ an invariant measure with a good product structure, then if 
$m$ is an $su$-invariant measure then there exists a disintegration that is $su/c$-invariant. 
\end{theorem}
\begin{proof}
For any topological space $N$, we denote by $\cM(N)$ the space of measures on $N$ with the weak$^*$ topology.

There exist $M^s$ and $M^u$ of total measure with $s$-invariance and $u$-invariance respectively. Take $M'=M^s\cap M^u$, $x\in \supp(\mu)$
and a foliated box $C_x\subset M$; via a local chart we can write $C=\Sigma\times D^c_R$ where $\Sigma$ is a transversal section to the center foliation and $D^c_R$ is a disc of radius $R$ 
in a center manifold, let $\pi:C_x\to \Sigma$ be the natural projection given by the center discs, observe that the center stable and center unstable manifolds gives a product structure of $\Sigma=\Sigma^s\times \Sigma^u$ and by hypothesis $\mu_{\Sigma}\sim\mu^s\times \mu^u$.

By the absolute continuity of the center foliation and the continuity of the Jacobians of the stable and unstable holonomies of $f$ we have that 
the disintegration $\Sigma\to \cM(D^c_R),\quad x\mapsto \mu^c_x$ is of the form 
$\mu^c_x=\rho \mu^c$ where $\mu^c$ is the Lebesgue measure in $D^c_R$ and $\rho$ is continuous.

We write $x\rels y$ for $x,y\in \Sigma$ in the same center stable manifold, and $x\relu y$ if they are in the same center unstable manifold.
Take $r>0$ smaller than $R$ such that for every $x\rels y$, $h^s_{x,y}:D^c_r\to D^c_R$ is well defined. 

Fix some $x^s\in \Sigma^s$ such that $\mu^u\left(\{x^s\}\times \Sigma^u\cap \pi(M')\right)=1$ and fix some $x^u\subset \Sigma^u$ such that $\mu^c_{(x^s,x^u)}((x^s,x^u)\times D^c_R\cap M')=1$, by the absolute continuity in the center direction this implies that Lebesgue almost every $t\in (x^s,x^u)\times D^c_R$ is in $M'$; for simplicity denote $x_0=(x^s,x^u)$.

Take $\epsilon>0$ such that 
$$
D^c_{\epsilon}\subset \bigcap_{z\relu y\rels x_0}h^u_{y,z} h^s_{x_0,y}(D^c_r).
$$

Lets call $m^c_x=\int_{D^c_r}m_t d\mu^c_x(t)$. Taking $x\rels y$ for $t\in D^c_r$ lets call $h(t)=h^s_{x,y}(t)$ and define $(\cH^s_{x,y}m^c_x)_{h(t)}={H^s_{(x,t),(y,h(t))}}_* m_t$. If $m_t$ is well defined for almost every $t\in D^c_r$
then by the absolute continuity of the holonomies $(\cH^s_{x,y}m^c_x)_{t'}$ is defined for almost every $t'\in h^s_{x,y}(D^c_r)$. Analogously we can define $(\cH^u_{x,y}m^c_x)_{t'}$

Now define a new disintegration of $m$ in the box $\Sigma\times D^c_\epsilon$ in the following way: 
Fix the restriction of the original disintegration defined for almost every $(x_0,t)\in \{x_0\}\times D^c_r$, (in particular $m^c_{x_0}$ is well defined) 
extend the disintegration to every $y\relu x_0$ and $t'\in  h^u_{x_0,y}(D^c_r)$ by $(\cH^u_{x_0,y}m^c_{x_0})_{t'}$. Lets call this disintegration $m^c_y$, and then extend it for every $z\rels y\relu x_0$ by
$(\cH^s_{y,z}m^c_y)_{t''}$ for almost every $t''\in h^s_{y,z} h^u_{x_0,y}(D^c_r)$. 
Now by definition of $\epsilon$ this disintegration is well defined in $\Sigma\times D^u_\epsilon$, and it coincides with the original disintegration almost everywhere, lets denote this
new disintegration by $t\mapsto m^s_t$. 

By construction this disintegration is $s/c$-invariant. We claim that this new disintegration, 
with respect to $\{x\}\times D^c_{\epsilon}\times \proj$, defined by 
$$
\Sigma\to \cM(D^c_r\times \proj),\quad m^{cs}_x=\int_{D^c_r}m^s_t d\mu^c_x(t),$$
is continuous.
Assuming this claim, we can define analogously, reducing $\epsilon$ if necessary, a disintegration $x\mapsto m^u_x$ that is $u/c$-invariant on $\Sigma\times \epsilon$. 
We have that $m^{cs}_x=m^{cu}_x$ for Lebesgue almost every $x\in \Sigma$, then as both are continuous we have the equality for every $x\in \Sigma$.
Then by the uniqueness of the Rokhlin disintegration, for every $x\in \Sigma$, $m^s_t=m^u_t$ for $\mu^c$ almost every $t\in D^c_{\epsilon}$. As this boxes cover $\supp(\mu)$ we conclude the theorem.

We are left to prove the claim.

We will prove that $x\mapsto m^{cs}_x$ is uniformly continuous varying $y\rels x$ and $z\relu y$.
Fix some continuous $\varphi:D^c_\epsilon \times \proj\to \real$ and let $y_n\to x$, $y_n\rels x$. 
\begin{equation*}
 \int \varphi (t,v) d m^{cs}_{y_n}=\int_{D^c_\epsilon} \int_{\proj} \varphi (t,v)dm^s_{y_n,t}(v) d\mu^c_{y_n}(t)
 \end{equation*}
\begin{equation*}
\int \varphi (t,v) d m^{cs}_{y_n}  =\int_{D^c_\epsilon} \int_{\proj} \varphi (t,v)d {H_{n,t}}_*m^s_{x,h_n(t)}(v) d\mu^c_{y_n}(t),
\end{equation*}
where $h_n(t)=h^s_{y_n,x}(t)$, and $H_{n,t}=H^s_{(x,h_n(t)),(y_n,t)}$ so we can write the integral as
$$
\int_{D^c_\epsilon} \int_{\proj} \varphi (t,H_{n,t} v)d m^s_{x,h_n(t)}(v) \rho(y_n,t) d\mu^c(t),
$$
changing variables and we have 
$$
\int_{h_n(D^c_\epsilon)} \int_{\proj} \varphi (h_n^{-1}(t),H_{n,h_n^{-1}(t)} v)d m^s_{x,t}(v) J_{\mu^c}h_n^{-1}(t) \rho(y_n,h_n^{-1}(t)) d\mu^c(t),
$$
where $ J_{\mu^c}h_n^{-1}$ is the Jacobian of $h_n^{-1}$ with respect to the Lebesgue measure $\mu^c$. Hence using that $h_n^{-1}\to \id_{W^c(x)}$, $J_{\mu^c}h_n^{-1}\to 1$, 
$H_{n,t}\to \id_{\proj}$ uniformly and the continuity of $\rho$ and $\varphi$, we conclude that $\int \varphi (t,v) d m^{cs}_{y_n}\to \int \varphi (t,v) d m^{cs}_{y}$.

For $z_n\relu x$, such that $z_n\to x$, we take $z_n\rels y_n\relu x_0$ and $x\rels x'\relu x_0$, by construction $y_n\to x'$ so 
$$
\int \varphi (t,v) d m^{cs}_{y_n}\to \int \varphi (t,v) d m^{cs}_{x'},
$$
observe that, 
$$
\int \varphi (t,v) d m^{cs}_{z_n}  =\int_{D^c_\epsilon} \int_{\proj} \varphi (t,v)d {H^u_{n,t}}_*m^s_{z_n,h^u_n(t)}(v) d\mu^c_{z_n}(t),
$$
where $H^u_{n,t}$ and $h^u_n$ are defined as before just changing stable by unstable in the corresponding point.
Therefore, using the same calculations as before, it follows that
$\int \varphi (t,v) d m^{cs}_{z_n}\to \int \varphi (t,v) d m^{cs}_{x}$.

This proves the claim and concludes the proof.
\end{proof}

As we mention before, the support of the SRB measure $\mu$ coincide with a non-trivial attractor of the spectral decomposition of $\Omega(\phi_t)$. This implies that we have density of periodic orbits and the homoclinic intersection of this periodic orbits is also dense and contained in the support of $\mu$.  

Now we prove that after a perturbation we have stably positive Lyapunov exponents for the cocycle. For this lets first find some center manifold with positive Lyapunov exponents.

Observe that, as the center manifolds of $f$ are the lines flows of $\phi_t$, the density of periodic orbits of the flow implies density of compact invariant center leaves.
Moreover, the restriction of $f$ to one of this compact leaves, $W^c(p)$, is a rotation in $\real/T\integer$ where $T$ is the period $p$ by the flow. So after a $C^1$ change of coordinates
we can suppose that $f_p=f\mid_{W^c(p)}=r_{\frac{1}{T}}$ where 
$$r_{\frac{1}{T}}:\real/\integer\to \real/\integer,\quad t\mapsto t+\frac{1}{T}.$$
The restriction of the cocycle to $W^c(p)$ is a cocycle over a rotation with Lebesgue $\mu^c_p$ as invariant measure.

Now we are going to define $h:W^c(p)\to W^c(p)$ that is a composition of stable and unstable holonomies (see figure~\ref{figure}).

Fix a flow periodic point $p\in \supp(\mu)$, lets take $z\in W^{cs}(p)\cap W^{cu}(p)$, by invariance of the center stable and center unstable manifolds $\phi_t(z)\in W^{cs}(p)\cap W^{cu}(p)$ for every $t\in \real$. Observe also that the orbit of $z$ is non recurrent.

\begin{figure}[ht]
    \centering
    \includegraphics[scale=0.6]{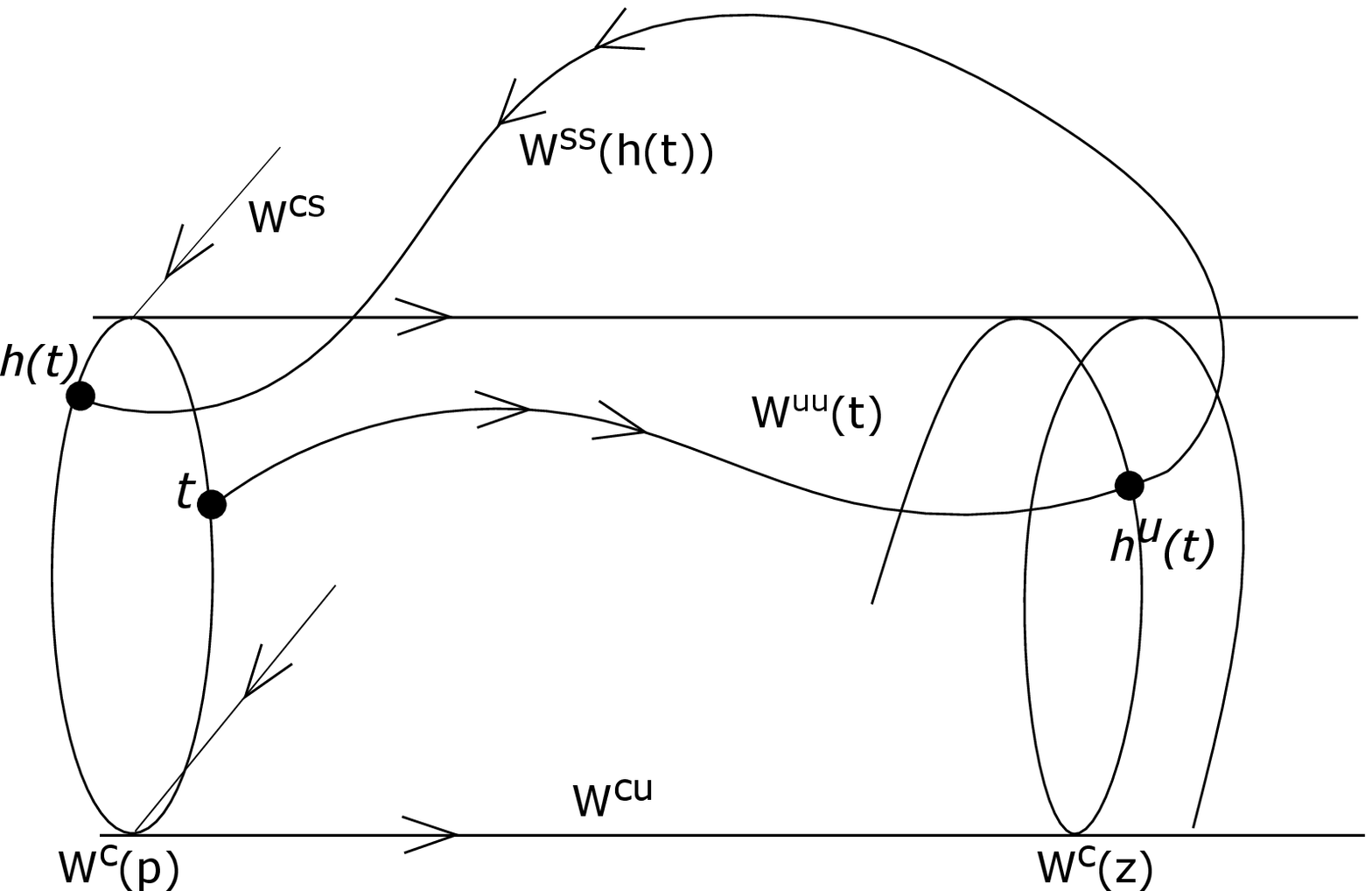}
    \caption{definition $h:W^c(p)\to W^c(p)$.}\label{figure}
\end{figure}

Suppose that actually $z\in W^u(p)$ and define 
$h(p)=W^c(p)\cap W^s(z)$, now for $t\in \left[ 1,T \right]$ define 
$h(\phi_t(p)):=\phi_t(h(p))$ and observe that by the invariance of the stable and unstable manifolds $\phi_t(h(p))\in W^c(\phi_t(p))\cap W^s(\phi_t(z))$ and $\phi_t(z)\in W^u(\phi_t(p))$. 

So $h$ is a composition of stable and unstable holonomies. Also in the circle coordinates, identifying $p$ with $0$,
$$
h:\real\integer\to \real \integer,\quad t\mapsto t+\omega,
$$ 
where $\omega$ is such that $\phi_\omega (p)=h(p)$.

Call $h^u_{p,z}:W^c(p)-\{p\}\to W^c(z)$ the map given by $t\mapsto \phi_t(z)$, by the same reasoning as before this map is given by an unstable holonomy, if there is no risk of ambiguity we just write $h^u=h^u_{p,z}$. This map is not well defined in $0$ because $\phi_T(z)\neq z$.

Observe that as the center stable and center unstable are dense in $\supp(\mu)$ we can find points $z_1,\dots,z_d$ as before such that any pair $z_i$, $z_j$, with $i\neq j$, are in different orbits of the flow. Hence, we can define maps $h_1,\dots,h_d$ with the properties above.

\begin{lemma}\label{l.posrest}
 There exists $A'$ arbitrarily close to $A$ such that for some center leave $W^c(p)$, $L(A_p,f_p,\mu^c_p)>0$.
\end{lemma}
\begin{proof}
 Suppose that there exists $p\in \supp(\mu)$ a periodic point for the flow such that $\frac{1}{T_p}$ is irrational then the Lebesgue measure is non periodic, so by \cite{Xu15}, 
 we can perturb $A$ such that $L(A_p,f_p,\mu^c_p)>0$.
 
 If for every periodic point $p\in \supp(\mu)$, $\frac{1}{T_p}$ is rational we have that as $T_p$ can be taken arbitrarily large, we have rational rotations with arbitrarily large period.
Now, taking 
$$
A_{\theta}=\left(\begin{array}{cc}
\cos(\theta)\id_d & \sin(\theta)\id_d\\
-\sin(\theta)\id_d &\cos(\theta)\id_d \ \\
\end{array}\right)
A,
$$ 
for a generic $A$ we have that for an $n$-periodic point there exists an $O(\frac{1}{n})$ dense set with $L(A_{\theta}(t))>0$ (see \cite[Lemma~3.4]{Xu15}),
so taking $p$ such that $T_p$ is very large we can take $\theta$ small such that $A_\theta$ is close to $A$.
 \end{proof}

We recall two results whose proofs can be found in \cite{Pol16}.
\begin{lemma}\label{l.open}
 Suppose there exists some sequence $A_k$ converging to $A$, such that $L(A_k,\mu)=0$ then $A$ admits some some $su$-invariant measure $m$.
\end{lemma}
\begin{lemma}\label{l.transv_finite}
Let $(E,\omega)$ be a $2d$-dimensional symplectic vector space and let $\{(V_j, W_j): 1\leq j\leq m\}$ be a finite collection of pairs subspaces of $E$ with
complementary dimensions (i.e., $\textrm{dim}(V_j) + \textrm{dim}(W_j) = 2d$ for all $1\leq j\leq m$). Then, there exists a symplectic automorphism $\sigma$ arbitrarily close 
to the identity such that
$$\sigma(V_j)\cap W_j =\{0\} \ \ \forall\, j=1, \dots, m.$$
\end{lemma}

As a consequence if the subspaces $V$ and $W$ have not complementary dimensions we can always make $\dim(\sigma(V)+W)=\dim V+\dim W$.

\begin{lemma}\label{l.nonsu}
 There exists $A'$ arbitrarily close to $A$ such that $A'$ does not admit any $su$-invariant measure.
\end{lemma}
\begin{proof}
 Given a neighbourhood of $A$ by Lemma~\ref{l.posrest} we can find $A'$ in this neighbourhood such that there exists some compact center leave $W^c(p)$ with $L(A'_p,f_p,\mu^c_p)>0$. 
 Take $z$ in $W^u(p)\cap W^{cs}(p)$ and $h:W^c(p)\to W^c(p)$ defined as before, observe that by the absolute continuity of the holonomies this map is absolutely continuous.
 
 Lets call the Lyapunov exponents of $(f_p,A'_p)$ by $\lambda^1_c,\dots,\lambda^k_c$. By the positivity of the integrated Lyapunov exponents there 
 exists $i$ such that $\lambda^i_c(t)>\lambda^{i+1}_c(t)$ for some 
 $1\leq i\leq 2d$, as we have finite possibilities of $i$ we can take some 
 positive measure set $I\subset W^c(p)$ such that, for every $t\in I$, there exists an invariant decomposition $\real^{2d}=V^u_t+V^s_t$ 
 with constant dimensions such that the smaller Lyapunov exponent in $V^u_t$ is larger than the largest Lyapunov exponent in $V^s_t$.
 
  We can also assume, reducing $I$ if necessary, that this decomposition varies continuously in $I$. As $h$ preserves the Lebesgue measure $\mu^c_p$ ($h$ is a rotation in the circle coordinates) we can find an iterate $j$ such that $h^j(I)\cap I$ has positive measure, take $j$ to be the smaller integer with this property.
  
Suppose that $A'$ admits some $su$-invariant measure, by Theorem~\ref{t.sucinv} we can take the disintegration to be $su/c$-invariant.
By \cite[Proposition~3.1]{BP15} we have that, for every $t\in I$, $m_t=m^u_t+m^s_t$ where $\supp(m^u_t)\subset V^u_t$ and $\supp(m^s_t)\subset V^s_t$. 
Taking $t'\in I$ density point of $h^j(I)\cap I$ we can find a neighbourhood $N\subset W^c(p)$ of $t'$ such that the sets $h^{u}_{p,z}(h^i(N))\subset W^c(z),i=0,\dots,j-1$ are disjoint. 

Let $H^{A'}_{t}$ be the composition of the stable linear holonomy from $t$ to the point $h^u(t)$ and the unstable linear 
holonomy from $h^u(t)$ to $h(t)$.

The $su/c$-invariance implies that $m_{h^j}(t)=(H^j_{t})_* m_t$ for $\mu^c_p$-almost every point. This implies that
$$
\supp(m_{h^j(t)})\subset \lbrace V^u_{h^j(t)},V^s_{h^j(t)}\rbrace\cap \lbrace H^j_{t}V^u_t,H^j_{t}V^s_t\rbrace.
$$

Fix a neighbourhood $\cN$ of $h^u(t')$ that does not contain $f^n(h^u(t'))$ for $n\neq 0$, making a small perturbation we find $\hA:M\to\Sp$ such that $\hA(h^u(t'))=\sigma A'(h^u(t'))$ and $\hA(h^u(r))=A'(h^u(r))$
for every $r\in W^c(z)-N$ and also $\hA(x)=A'(x)$ outside $\cN$. We have that 
$${H^{\hA}_{t'}}^j=H^{j-1,A'}_{h({t'})}\circ H^{s,A'}_{h^u({t'}),h({t'})}\circ \sigma \circ H^{u,A'}_{{t'},h^u({t'})}.
$$
Then using \ref{l.transv_finite} we can find $\sigma$ arbitrarily close to the identity such that 
$V^s_{h^j(t')}\cap H^j_{t'}V^u_{t'}=\emptyset$ and $V^u_{h^j(t')}\cap H^j_{t'}V^s_{t'}=\emptyset$. So

$$
\supp(m_{h^j(t')})\subset V'_{h^j(t')}=V^s_{h^j(t)}\cap H^j_{t'}V^s_{t'}
$$
using \ref{l.transv_finite} again we can make $\dim V'<\dim V^s$. As $t'$ and $h^j(t')$ are density points of a set where the $V^u$ and $V^s$ varies continuously then there exists a positive measure subset of $t\in I$ such that $$
\supp(m_{h^j(t)})\subset V'_{h^j(t)}\quand \dim V'<\dim V^s.
$$ 

\begin{figure}[ht]
    \centering
    \includegraphics[scale=0.6]{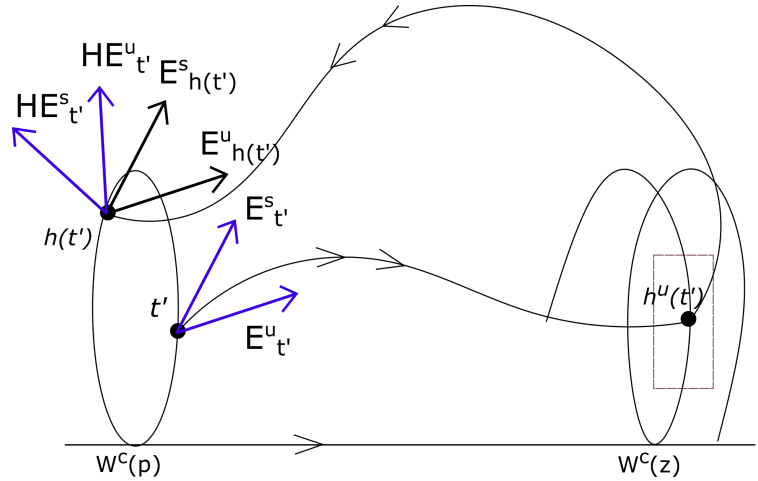}
    \caption{perturbation of $H$}\label{fig.perturb}
    \end{figure}

If $\hA$ still admits some $m$ $su$-invariant, we can repeat the argument taking a different homoclinic point $z_1,\dots,z_d$ (i.e the orbit of $z_i$ is not in the orbit of $z_j$ for $i\neq j$) with $V'$ instead of $V^s$ such that the new perturbation does not affect the cocycle in the orbit of $z$. Inductively, if the perturbed cocycles always admit some $su$-invariant measure, we can do this with
less than $d$-point $z_i$ to get that $\supp(m_t)=\emptyset$, for $t$ in a positive measure set, a contradiction. 

Then we can find $A''$ arbitrarily close to $A$ does not admit any $su$-invariant measure.
 \end{proof}
\begin{proof}[proof of Theorem~\ref{teo1}]
 By Lemma~\ref{l.nonsu} the set of cocycles that do not admit any $su$-invariant measure is dense, by the invariance principle~\cite{Extremal} and by Lemma~\ref{l.open} 
 every cocycle in this set has a neighbourhood of cocycles with positive Lyapunov exponents. 
\end{proof}

\bibliography{bib}
\bibliographystyle{plain}

\end{document}